\documentclass[11pt]{article}
\usepackage{amssymb,amsmath}
\usepackage[mathscr]{eucal}
\usepackage[cm]{fullpage}
\usepackage[english]{babel}
\usepackage[latin1]{inputenc}
\usepackage{color}

\def\dom{\mathop{\mathrm{Dom}}\nolimits}
\def\N{\mathbb N}
\def\PT{\mathcal{P}\mathcal{T}}
\def\T{\mathcal{T}}
\def\Sym{\mathcal{S}}

\def\C{\mathcal{C}}
\def\D{\mathcal{D}}
\def\DPC{\mathcal{D}\mathcal{P}\mathcal{C}}
\def\POPI{\mathcal{P}\mathcal{O}\mathcal{P}\mathcal{I}}
\def\PORI{\mathcal{P}\mathcal{O}\mathcal{R}\mathcal{I}}
\def\OP{\mathcal{O}\mathcal{P}}
\def\OR{\mathcal{O}\mathcal{R}}
\def\POP{\mathcal{P}\mathcal{O}\mathcal{P}}
\def\POR{\mathcal{P}\mathcal{O}\mathcal{R}}
\def\I{\mathcal{I}}
\newtheorem{theorem}{Theorem}[section]
\newtheorem{proposition}[theorem]{Proposition}
\newtheorem{corollary}[theorem]{Corollary}
\newtheorem{lemma}[theorem]{Lemma}
\newenvironment{proof}{\begin{trivlist}\item[\hskip%
\labelsep{\bf Proof.}]}%
{\qed\rm\end{trivlist}}
\newcommand{\qed}{{\unskip\nobreak
\hfil\penalty50\hskip .001pt \hbox{}
          \nobreak\hfil
         \vrule height 1.2ex width 1.1ex depth -.1ex
           \parfillskip=0pt\finalhyphendemerits=0\medbreak}}

\newcommand{\lastpage}{\addresss}

\newcommand{\addresss}{\small \sf  
\noindent{\sc V\'\i tor H. Fernandes}, 
Center for Mathematics and Applications (NovaMath)  
and Department of Mathematics, FCT NOVA, 
Faculdade de Ci\^encias e Tecnologia, 
Universidade Nova de Lisboa, 
Monte da Caparica, 
2829-516 Caparica, 
Portugal; 
e-mail: vhf@fct.unl.pt. 
}

\title{Oriented transformations on a finite chain: another description}

\author{V\'\i tor H. Fernandes\footnote{This work is funded by national funds through the FCT - Funda\c c\~ao para a Ci\^encia e a Tecnologia, I.P., under the scope of the projects UIDB/00297/2020 and UIDP/00297/2020 (NovaMath - Center for Mathematics and Applications).}~
}

\begin{document}

\maketitle

\begin{abstract} 
Following the new description of an oriented full transformation on a finite chain given recently by Higgins and Vernitski in \cite{Higgins&Vernitski:2022}, 
in this short note we present a refinement of this description which is extendable to partial transformations and to injective partial transformations.
\end{abstract}

\medskip



\section*{Introduction}\label{presection} 

For $n\in\N$, let $\Omega_n$ be a set with $n$ elements, e.g. $\Omega_n=\{1,2,\ldots,n\}$.  
As usual, denote by $\PT_n$ the monoid (under composition) of all 
partial transformations on $\Omega_n$, 
by $\T_n$ the submonoid of $\PT_n$ of all full transformations on $\Omega_n$, 
by $\I_n$ the \textit{symmetric inverse monoid} on $\Omega_n$, i.e. 
the inverse submonoid of $\PT_n$ of all 
partial permutations on $\Omega_n$, 
and by $\Sym_n$ the \textit{symmetric group} on $\Omega_n$, 
i.e. the subgroup of $\PT_n$ of all 
permutations on $\Omega_n$. 

Next, suppose that $\Omega_n$ is a chain, e.g. $\Omega_n=\{1<2<\cdots<n\}$. 
Let $s=(a_1,a_2,\ldots,a_t)$
be a sequence of $t$ ($t\geqslant0$) elements
from the chain $\Omega_n$. 
We say that $s$ is \textit{cyclic} 
[\textit{anti-cyclic}] if there
exists no more than one index $i\in\{1,\ldots,t\}$ such that
$a_i>a_{i+1}$ [$a_i<a_{i+1}$],
where $a_{t+1}$ denotes $a_1$.
Notice that, the sequence $s$ is cyclic
[anti-cyclic] if and only if $s$ is empty or there exists
$i\in\{0,1,\ldots,t-1\}$ such that 
$a_{i+1}\leqslant a_{i+2}\leqslant \cdots\leqslant a_t\leqslant a_1\leqslant \cdots\leqslant a_i $ 
[$a_{i+1}\geqslant a_{i+2}\geqslant \cdots\geqslant a_t\geqslant a_1\geqslant \cdots\geqslant a_i $] (the index
$i\in\{0,1,\ldots,t-1\}$ is unique unless $s$ is constant and
$t\geqslant2$). We also say that $s$ is \textit{oriented} if $s$ is cyclic or $s$ is anti-cyclic.  
Given a partial transformation $\alpha\in\PT_n$ such that
$\dom(\alpha)=\{a_1<\cdots<a_t\}$, with $t\geqslant0$, we 
say that $\alpha$ is \textit{orientation-preserving} 
[\textit{orientation-reversing}, \textit{oriented}] if the sequence of its images
$(a_1\alpha,\ldots,a_t\alpha)$ is cyclic [anti-cyclic, oriented].  
We denote by $\POP_n$ the submonoid of $\PT_n$ 
of all orientation-preserving partial
transformations and by $\POR_n$ the
submonoid of $\PT_n$ of all
oriented  partial transformations. 
Let $\OP_n$ be the submonoid of $\PT_n$ 
of all orientation-preserving full transformations and $\OR_n$ be the
submonoid of $\PT_n$ of all oriented full transformations, 
i.e. $\OP_n=\POP_n\cap\T_n$ and $\OR_n=\POR_n\cap\T_n$. 
Let us also consider the injective partial counterparts of $\OP_n$ and $\OR_n$, 
i.e. the inverse submonoids $\POPI_n=\POP_n\cap\I_n$ and $\PORI_n=\POR_n\cap\I_n$ of $\PT_n$. 

The notion of an orientation-preserving full transformation was introduced by McAlister in \cite{McAlister:1998} and, independently, by Catarino and Higgins in \cite{Catarino&Higgins:1999}. The partial and injective partial versions of this concept were first considered by the author in \cite{Fernandes:2000}. 
Since then, many articles have been published by several authors involving semigroups of oriented transformations. 
In particular, the author, with various different co-authorships, has more than a dozen published papers on these semigroups. 

Recently Higgins and Vernitski presented in \cite{Higgins&Vernitski:2022} a new description of an oriented transformation. 
From this description, in this short note, we derive yet another description of an oriented transformation which is, 
on the one hand and in a certain sense, simpler and, on the other hand, extendable to partial transformations and to injective partial transformations.

\section{The description} \label{description} 

Let us consider the following permutations of $\Omega_n$ of order $n$ and $2$, respectively: 
$$
g=\begin{pmatrix} 
1&2&\cdots&n-1&n\\
2&3&\cdots&n&1
\end{pmatrix} 
\quad\text{and}\quad  
h=\begin{pmatrix} 
1&2&\cdots&n-1&n\\
n&n-1&\cdots&2&1
\end{pmatrix}. 
$$ 
It is clear that $g$ is an orientation-preserving transformation and $h$ is an orientation-reversing transformation. 
Moreover, for $n\geqslant3$, $g$ together with $h$ generate the well-known \textit{dihedral group} $\D_{2n}$ of order $2n$ 
(considered as a subgroup of $\Sym_n$): 
$$
\D_{2n}=\{1,g,g^2,\ldots,g^{n-1}, h,hg,hg^2,\ldots,hg^{n-1}\}. 
$$
Let $\C_n$ be the \textit{cyclic group} of order $n$ generated by $g$, i.e. 
$\C_n=\{1,g,g^2,\ldots,g^{n-1}\}$.  

Notice that $\C_n$ is the group of units of $\OP_n$, $\POPI_n$ and $\POP_n$ and, on the other hand, 
for $n\geqslant3$, $\D_{2n}$ is the group of units of $\OR_n$, $\PORI_n$ and $\POR_n$ (for $n\in\{1,2\}$, their group of units is also $\C_n$). 

\smallskip 

Now, let $s=(a_1,a_2,\ldots,a_t)$
be a sequence of $t$ elements
from the chain $\Omega_n$, with $t\geqslant3$. It is easy to show that: 
\begin{itemize}
\item $s$ is cyclic [anti-cyclic] if and only if, for all $\sigma\in\C_t$,  $(a_{1\sigma},a_{2\sigma},\ldots,a_{t\sigma})$ is cyclic [anti-cyclic]; 
\item $s$ is cyclic [anti-cyclic] if and only if, for all $\sigma\in\D_{2t}\setminus\C_t$,  $(a_{1\sigma},a_{2\sigma},\ldots,a_{t\sigma})$ is anti-cyclic [cyclic]; 
\item $s$ is oriented if and only if, for all $\sigma\in\D_{2t}$,  $(a_{1\sigma},a_{2\sigma},\ldots,a_{t\sigma})$ is oriented; 
\item $s$ is cyclic [anti-cyclic, oriented] if and only if there exists $\sigma\in\C_t$ [$\sigma\in\D_{2t}\setminus\C_t$, $\sigma\in\D_{2t}$] such that $a_{1\sigma}\leqslant a_{2\sigma}\leqslant \cdots\leqslant a_{t\sigma}$. 
\end{itemize}

Next, we recall the following recent descriptions of $\OP_n$ and $\OR_n$ 
proved by Higgins and Vernitski in \cite[Theorem 3 and Theorem 7]{Higgins&Vernitski:2022} (see also \cite{Levi&Mitchell:2006}): 

\begin{proposition}\label{higver}
Let $\alpha\in\T_n$. Then:
\begin{enumerate}
\item $\alpha\in\OP_n$ if and only if, for every triple $(a_1,a_2,a_3)$ of elements of $\Omega_n$, $(a_1,a_2,a_3)$ and 
$(a_1\alpha,a_2\alpha,a_3\alpha)$ are both cyclic or both anti-cyclic; 
\item $\alpha\in\OR_n$ if and only if, for every oriented quadruple $(a_1,a_2,a_3,a_4)$ of elements of $\Omega_n$, 
the quadruple $(a_1\alpha,a_2\alpha,a_3\alpha,a_4\alpha)$ is also oriented. 
\end{enumerate}
\end{proposition} 

Observe that it is easy to show that any triple of elements of $\Omega_n$ is oriented. 
Furthermore, given a triple $(a_1,a_2,a_3)$ of elements of $\Omega_n$,  $(a_1,a_2,a_3)$ 
is cyclic if and only if $(a_3,a_2,a_1)$ is anti-cyclic. 
This allows us to easily derive from Property 1 of Proposition \ref{higver} the following simpler characterization of $\OP_n$. 

\begin{corollary}\label{higverfer}
Let $\alpha\in\T_n$. Then 
$\alpha\in\OP_n$ if and only if, for every cyclic triple $(a_1,a_2,a_3)$ of elements of $\Omega_n$, $(a_1\alpha,a_2\alpha,a_3\alpha)$ is also cyclic.  
\end{corollary} 

In fact, we can further simplify both characterizations given in Proposition \ref{higver}: 

\begin{corollary}\label{ferhigver}
Let $\alpha\in\T_n$. Then 
 $\alpha\in\OP_n$ $[\alpha\in\OR_n]$ if and only if, for every non-decreasing triple $[$quadruple$]$ $(a_1,a_2,a_3[,a_4])$ of elements of $\Omega_n$, 
the triple $[$quadruple$]$ $(a_1\alpha,a_2\alpha,a_3\alpha [,a_4\alpha])$ is cyclic $[$oriented$]$. 
\end{corollary}
\begin{proof}
Since any non-decreasing sequence of $\Omega_n$ is cyclic, for both properties, it remains to prove the converse implications. 

Suppose that, for every non-decreasing triple [quadruple]  $(a_1,a_2,a_3[,a_4])$ of elements of $\Omega_n$, the triple [quadruple] 
$(a_1\alpha,a_2\alpha,a_3\alpha[,a_4\alpha])$ is cyclic [oriented]. 
Let $(a_1,a_2,a_3[,a_4])$ be a cyclic triple [an oriented quadruple] of elements of $\Omega_n$. 
Take $\sigma\in\C_3$ [$\sigma\in\D_{2\cdot4}$] such that 
$a_{1\sigma}\leqslant a_{2\sigma}\leqslant a_{3\sigma}[\leqslant a_{4\sigma}]$. 
Then, by hypothesis, $(a_{1\sigma}\alpha,a_{2\sigma}\alpha, a_{3\sigma}\alpha[,a_{4\sigma}\alpha])$ is cyclic [oriented]. 
Let $a'_i=a_i\alpha$, for $i=1,2,3[,4]$. Then $(a'_{1\sigma},a'_{2\sigma}, a'_{3\sigma}[,a'_{4\sigma}])$ is cyclic [oriented] and so,  
since $\sigma^{-1}\in\C_3$ [$\sigma^{-1}\in\D_{2\cdot4}$], by the above observation,  
$$
(a_1\alpha,a_2\alpha,a_3\alpha[,a_4\alpha])=(a'_1,a'_2, a'_3[,a'_4])=
(a'_{(1\sigma)\sigma^{-1}},a'_{(2\sigma)\sigma^{-1}}, a'_{(3\sigma)\sigma^{-1}}[,a'_{(4\sigma)\sigma^{-1}}])
$$ 
is also cyclic [oriented]. Thus, by Corollary \ref{higverfer} [Proposition \ref{higver}], 
we get $\alpha\in\OP_n$ $[\alpha\in\OR_n]$, as required.
\end{proof} 

As observed above, any triple of elements of $\Omega_n$ is oriented. From this fact, it is easy to deduce 
that any quadruples of elements of $\Omega_n$ of the form $(a,a,b,c)$, $(a,b,b,c)$ and $(a,b,c,c)$ are also oriented. 
It is also clear that triples of $\Omega_n$ of the form $(a,a,b)$ and $(a,b,b)$ are cyclic. 
Therefore, it is easy to check that we can replace non-decreasing triples and quadruples by  (strictly) increasing triples and quadruples, respectively, in the statements of Corollary \ref{ferhigver} and, in this way, obtaining the following characterizations of $\OP_n$ and $\OR_n$,
where by \textit{width} of a partial transformation we mean, as usual, the number of elements in its domain: 

\begin{theorem}\label{charopnorn} 
Let $\alpha\in\T_n$. Then $\alpha\in\OP_n$  $[\alpha\in\OR_n]$ if and only if every restriction of $\alpha$ 
of width three $[$four$]$ belongs to $\POP_n$ $[\POR_n]$.  
\end{theorem} 

Next, we aim to extend this last result to partial transformations. 

Let $\alpha\in\PT_n$. For our purpose, it is enough to consider transformations of width greater than or equal to three. 
Thus, let us suppose that $\dom(\alpha)=\{i_1<i_2<\cdots<i_k\}$, with $k\geqslant3$. Then, we define 
$\bar\alpha\in\T_n$ by 
$$
i\bar\alpha = \left\{\begin{array}{ll}
i_1\alpha & \mbox{if $1\leqslant i\leqslant i_2-1$}\\
i_\ell\alpha & \mbox{if $i_\ell\leqslant i < i_{\ell+1}$, $\ell=2,3,\ldots,k-1$}\\
i_k\alpha & \mbox{if $i_k\leqslant i\leqslant n$}. 
\end{array}\right. 
$$

The following lemma is easy to check:

\begin{lemma}\label{lemcharpopnporn}
Let $\alpha\in\PT_n$ be such that $|\dom(\alpha)|\geqslant3$. Then $\alpha\in\POP_n$  $[\alpha\in\POR_n]$ if and only if  $\bar\alpha\in\OP_n$  $[\bar\alpha\in\OR_n]$.  
\end{lemma}

Now, take $\alpha\in\PT_n$ such that every restriction of $\alpha$ of width three [four] belongs to $\POP_n$ [$\POR_n$].  

Our objective is to prove that $\alpha\in\POP_n$  $[\alpha\in\POR_n]$. 
Take $\dom(\alpha)=\{i_1<i_2<\cdots<i_k\}$. 

If $|\dom(\alpha)|\leqslant2$ [$\leqslant3]$ then it is clear that $\alpha\in\POP_n$  $[\alpha\in\POR_n]$. 
Therefore, we are going to consider $|\dom(\alpha)|\geqslant3$ [$\geqslant4]$. 

Define $I_1=\{1,\ldots,i_2-1\}$, $I_\ell=\{i_\ell,\ldots,i_{\ell+1}-1\}$, for $\ell=2,\ldots,k-1$, and $I_k=\{i_k,\ldots, n\}$. 

Let $X=\{a_1<a_2<a_3[<a_4]\}$ be a subset of $\dom(\alpha)$ with three [four] elements. 
Then, we aim to show that $\bar\alpha|_X\in\POP_n$ [$\bar\alpha|_X\in\POR_n$]. 

Take $r_1,r_2,r_3 [,r_4] \in\{1,2,\ldots,k\}$ such that $a_i\in I_{r_i}$, for $i=1,2,3[,4]$. 
Since $a_1<a_2<a_3[<a_4]$, then $r_1\leqslant r_2\leqslant r_3[\leqslant r_4]$ and 
$$
\bar\alpha|_X=\left(
\begin{array}{ccc} 
a_1&a_2&a_3\\
i_{r_1}\alpha & i_{r_2}\alpha & i_{r_3}\alpha 
\end{array}
\left[\begin{array}{c} 
a_4\\
i_{r_4}\alpha
\end{array}\right]
\right). 
$$
If $r_1=r_2$ or $r_2=r_3$ [or $r_3=r_4$] then $i_{r_1}\alpha=i_{r_2}\alpha$ or $i_{r_2}\alpha=i_{r_3}\alpha$ 
[or $i_{r_3}\alpha=i_{r_4}\alpha$], whence $(i_{r_1}\alpha,i_{r_2}\alpha,i_{r_3}\alpha [,i_{r_4}\alpha])$ is cyclic [oriented] 
and so $\bar\alpha|_X\in\POP_n$ [$\bar\alpha|_X\in\POR_n$]. 

Otherwise, we have $r_1<r_2<r_3 [< r_4]$ and thus, by hypothesis, 
$$
\alpha|_{\{i_{r_1},i_{r_2},i_{r_3}\}}\in\POP_n \quad [\alpha|_{\{i_{r_1},i_{r_2},i_{r_3},i_{r_4}\}}\in\POR_n]. 
$$
 Hence $(i_{r_1}\alpha,i_{r_2}\alpha,i_{r_3}\alpha [,i_{r_4}\alpha])$ is cyclic [oriented] 
 and so, also in this case, we obtain $\bar\alpha|_X\in\POP_n$ [$\bar\alpha|_X\in\POR_n$]. 
 
 Now, by applying Theorem \ref{charopnorn}, we get $\bar\alpha\in\OP_n$  $[\bar\alpha\in\OR_n]$ and then,  
 by Lemma \ref{lemcharpopnporn}, it follows that $\alpha\in\POP_n$  $[\alpha\in\POR_n]$. 

\smallskip 

Thus, we may conclude the following characterizations of $\POP_n$ and $\POR_n$: 

\begin{theorem}\label{charpopnporn}
Let $\alpha\in\PT_n$. Then $\alpha\in\POP_n$  $[\alpha\in\POR_n]$ if and only if every restriction of $\alpha$ 
of width three $[$four$]$ belongs to $\POP_n$ $[\POR_n]$.  
\end{theorem} 

And, as an immediate corollary of the previous result, we have: 

\begin{theorem}\label{charpopinporin}
Let $\alpha\in\I_n$. Then $\alpha\in\POPI_n$  $[\alpha\in\PORI_n]$ if and only if every restriction of $\alpha$ 
of width three $[$four$]$ belongs to $\POPI_n$ $[\PORI_n]$.  
\end{theorem} 

\section{An application}\label{application} 

Let $n\geqslant3$.  Consider the \textit{cycle graph} 
$
C_n=(\{1,2,\ldots, n\}, \{\{i,i+1\}\mid i=1,2,\ldots,n-1\}\cup\{\{1,n\}\}) 
$
with $n$ vertices.  
Denote by $d(x,y)$ the (\textit{geodesic}) \textit{distance} between two vertices $x$ and $y$ of $C_n$, 
i.e. the number of edges in a shortest path between $x$ and $y$. 
Notice that $d(x,y)=\min \{|x-y|,n-|x-y|\}$  and so $0\leqslant d(x,y)\leqslant\frac{n}{2}$,
for all $x,y \in \{1,2,\ldots,n\}$. 
Let us consider the monoid $\DPC_n$ of all 
\textit{partial isometries} (or \textit{distance preserving partial transformations}) of $C_n$, i.e. 
$
\DPC_n=\{\alpha\in\PT_n\mid\mbox{$d(x\alpha,y\alpha)=d(x,y)$, for all $x,y\in\dom(\alpha)$}\}. 
$
This monoid was studied by the author together with Paulista in \cite{Fernandes&Paulista:2022arxiv}.  
Observe that it is easy to show that $\DPC_n$ is an inverse submonoid of the symmetric inverse monoid $\I_n$. 
Furthermore, $\DPC_n$ is also a submonoid of $\PORI_n$. 
However, this last inclusion is not trivial. Although it has already been proved in the above mentioned paper, 
we aim to present here an alternative proof of this property by making use of Theorem \ref{charpopinporin}: 

\begin{proposition}\label{dpcpopi}
The monoid $\DPC_n$ is contained in $\PORI_n$. 
\end{proposition}
\begin{proof}
By Theorem \ref{charpopinporin}, it suffices to show that all transformations of $\DPC_n$ of width $4$ belong to $\PORI_n$ (notice that any restriction of an element of $\DPC_n$ is obviously an element of $\DPC_n$).  

Let 
$$
\alpha=\begin{pmatrix} 
i_1&i_2&i_3&i_4\\
j_1&j_2&j_3&j_4
\end{pmatrix} \in\DPC_n, 
$$
with $i_1<i_2<i_3<i_4$. 
Let $\alpha'=g^{i_1-1}\alpha g^{n-j_1+1}\in\DPC_n$. 
Then 
$
1\alpha'=1g^{i_1-1}\alpha g^{n-j_1+1}=i_1\alpha g^{n-j_1+1}=j_1 g^{n-j_1+1}=1
$ 
and, since $\alpha=g^{n-i_1+1}\alpha' g^{j_1-1}$  and $g\in\PORI_n$, we also have 
$\alpha\in\PORI_n$ if and only if $\alpha'\in\PORI_n$.  
Therefore, we can reduce the proof to transformations of $\DPC_n$ of width $4$ of the form 
$$
\alpha=\begin{pmatrix} 
1&i_2&i_3&i_4\\
1&j_2&j_3&j_4
\end{pmatrix},  
$$
with $i_2<i_3<i_4$. 
Clearly, in this case, $\alpha\not\in\PORI_n$ if and only if either 
($j_2<j_3$ and $j_3>j_4$) or ($j_2>j_3$ and $j_3<j_4$). 
Let 
$$
\alpha'=\alpha hg = 
\begin{pmatrix} 
1&i_2&i_3&i_4\\
1&n-j_2+2&n-j_3+2&n-j_4+2
\end{pmatrix}.  
$$
Since $g,h\in\PORI_n$ and $\alpha=\alpha' g^{n-1}h$, 
we have $\alpha\in\PORI_n$ if and only if $\alpha'\in\PORI_n$. 
On the other hand, we also have 
$j_2<j_3$ and $j_3>j_4$ if and only if $n-j_2+2>n-j_3+2$ and $n-j_3+2<n-j_4+2$ 
and, dually, 
$j_2>j_3$ and $j_3<j_4$ if and only if $n-j_2+2<n-j_3+2$ and $n-j_3+2>n-j_4+2$. 
Thus, we can further reduce the proof to transformations of $\DPC_n$ of width $4$ of the form 
$$
\alpha=\begin{pmatrix} 
1&i_2&i_3&i_4\\
1&j_2&j_3&j_4
\end{pmatrix},  
$$
with $i_2<i_3<i_4$ and $j_2<j_3$. 

In order to obtain a contradiction, suppose that $\alpha\not\in\PORI_n$. Then, we must also have $j_3>j_4$. 

\smallskip 

Let $p=2,3,4$. Then, from the equality $d(1,i_p)=d(1,j_p)$, we obtain 
$$
j_p=\left\{\begin{array}{ll}
i_p & \mbox{if either $i_p-1\leqslant\frac{n}{2}$ and $j_p-1\leqslant\frac{n}{2}$ or 
                               $i_p-1>\frac{n}{2}$ and $j_p-1>\frac{n}{2}$} \\ \\
n-i_p+2 &   \mbox{if either $i_p-1\leqslant\frac{n}{2}$ and $j_p-1>\frac{n}{2}$ or 
                               $i_p-1>\frac{n}{2}$ and $j_p-1\leqslant\frac{n}{2}$}.                             
\end{array}\right. 
$$

Now, we proceed by considering all possible cases for $j_2$, $j_3$ and $j_4$. 

\smallskip 

Case 1: $j_2=i_2$. 

\smallskip 

Case 1.1: $j_2=i_2$ and $j_3=i_3$. 

Since $j_4<j_3$, we can not have $j_4=i_4$ and so $j_4=n-i_4+2$, i.e. 
$
\alpha=\begin{pmatrix} 
1&i_2&i_3&i_4\\
1&i_2&i_3&n-i_4+2
\end{pmatrix}
$. 
Notice that $n-i_4+2<i_3$. 

First, admit that $i_4-i_3\leqslant\frac{n}{2}$, whence $d(i_3,i_4)=i_4-i_3$. 
If $d(n-i_4+2,i_3)=n-i_3+(n-i_4+2)$ then $i_4-i_3=n-i_3+n-i_4+2$ and so $i_4=n+1$, which is a contradiction. 
Hence, $d(n-i_4+2,i_3)=i_3-(n-i_4+2)$ and so $i_4-i_3=i_3-n+i_4-2$, which implies that $i_3=\frac{n+2}{2}$. 

On the other hand, admit that $i_4-i_3>\frac{n}{2}$, whence $d(i_3,i_4)=n-i_4+i_3$. 
If $d(n-i_4+2,i_3)=i_3-(n-i_4+2)$ then $n-i_4+i_3=i_3-n+i_4-2$ and so $i_4=n+1$, which is a contradiction. 
Hence, $d(n-i_4+2,i_3)=n-i_3+(n-i_4+2)$ and so $n-i_4+i_3=n-i_3+n-i_4+2$, which implies that $i_3=\frac{n+2}{2}$. 

Therefore, in either case, we have $i_3=\frac{n+2}{2}$. 

Next, suppose that $i_4-i_2\leqslant\frac{n}{2}$, whence $d(i_2,i_4)=i_4-i_2$. If $d(i_2,n-i_4+2)=n-|(n-i_4+2)-i_2|$ then 
$i_2=1$ or $i_4=n+1$, 
which is a contradiction in both cases. Hence, $d(i_2,n-i_4+2)=|(n-i_4+2)-i_2|$ and so 
$i_4=\frac{n+2}{2}$ or $i_2=\frac{n+2}{2}$, 
which is again a contradiction (since $i_3=\frac{n+2}{2}$). 

Therefore, $i_4-i_2>\frac{n}{2}$ and so $d(i_2,i_4)=n-i_4+i_2$. 
If $d(i_2,n-i_4+2)=|(n-i_4+2)-i_2|$ then 
$i_2=1$ or $i_4=n+1$,
which is a contradiction in both cases. Hence, $d(i_2,n-i_4+2)=n-|(n-i_4+2)-i_2|$ and so 
$i_4=\frac{n+2}{2}$ or $i_2=\frac{n+2}{2}$, 
which is again a contradiction (since $i_3=\frac{n+2}{2}$). 

Thus, Case 1.1 cannot occur. 

\smallskip 

Case 1.2: $j_2=i_2$ and $j_3=n-i_3+2$. 

We begin by observing that, since $i_2<n-i_3+2$, from the equality $d(i_2,i_3)=d(i_2,n-i_3+2)$, we have 
$i_2=1$ or $i_3=\frac{n+2}{2}$, 
whence $i_3=\frac{n+2}{2}$. 

If $j_4=n-i_4+2$ then, from the equality $d(i_2,i_4)=d(i_2,n-i_4+2)$, by the same calculations made above, we obtain 
$i_2=\frac{n+2}{2}$ or $i_4=\frac{n+2}{2}$, which is a contradiction (since $i_3=\frac{n+2}{2}$). 
Therefore, $j_4=i_4>i_3=\frac{n+2}{2}=n-i_3+2=j_3$,  which is again a contradiction.
 
Thus, Case 1.2 cannot occur either and so Case 1 does not occur.

\smallskip 

Case 2: $j_2=n-i_2+2$. 

As $i_2<i_3$, then $n-i_3+2<n-i_2+2$ and so $j_3=i_3$ (since $j_2<j_3$). Now, since $n-i_2+2<i_3$, 
from the equality $d(i_2,i_3)=d(n-i_2+2,i_3)$, we have 
$i_3=n+1$ or $i_2=\frac{n+2}{2}$, 
whence $i_2=\frac{n+2}{2}=n-i_2+2=j_2$. 
Therefore, this case is a particular instance of Case 1 and so cannot occur either.

\smallskip 

Thus, we conclude that $\alpha\in\PORI_n$, as required. 
\end{proof}

\bigskip 

\lastpage 

\end{document}


\maketitle

\begin{abstract} 
In this note, we aim to correct some of the results presented in \cite{Fernandes:2023}. 
Namely, the statements of Proposition 2.1, Corollary 2.2, Corollary 2.3, Theorem 2.4 and Theorem 2.6, concerning only the monoids $\OP_n$ and $\POP_n$, 
have to exclude transformations of rank two. 
All other results of \cite{Fernandes:2023}, as well as those mentioned above but for the monoids $\OR_n$ and $\POR_n$, do not require correction. 
\end{abstract}

\medskip

\noindent{\small 2020 \it Mathematics subject classification: \rm 20M20.} 

\noindent{\small\it Keywords: \rm orientation-preserving transformations, orientation-reversing transformations, oriented transformations.} 

\bigskip 

\bigskip 

We begin this note by briefly recalling some notions and by establishing some notations. 
Let $n$ be a positive integer and let $\Omega_n=\{1<\cdots<n\}$.  
Denote by $\PT_n$ the monoid (under composition) of all 
partial transformations on $\Omega_n$, 
by $\T_n$ the submonoid of $\PT_n$ of all full transformations on $\Omega_n$ and  
by $\I_n$ the inverse submonoid of $\PT_n$ of all 
partial permutations on $\Omega_n$. 
Let $s=(a_1,\ldots,a_t)$
be a sequence of $t$ ($t\geqslant0$) elements
from the chain $\Omega_n$. 
We say that $s$ is \textit{cyclic} 
[\textit{anti-cyclic}] if there
exists no more than one index $i\in\{1,\ldots,t\}$ such that
$a_i>a_{i+1}$ [$a_i<a_{i+1}$],
where $a_{t+1}$ denotes $a_1$. 
Given a partial transformation $\alpha\in\PT_n$ such that
$\dom(\alpha)=\{a_1<\cdots<a_t\}$, with $t\geqslant0$, we 
say that $\alpha$ is \textit{preserves orientation} 
[\textit{reverses orientation}] if the sequence of its images
$(a_1\alpha,\ldots,a_t\alpha)$ is cyclic [anti-cyclic].  
We denote by $\POP_n$ the submonoid of $\PT_n$ 
of all partial transformations that preserve orientation and by $\POR_n$ the
submonoid of $\PT_n$ of all partial transformations that preserve or reverse orientation. 
Let $\OP_n=\POP_n\cap\T_n$, $\OR_n=\POR_n\cap\T_n$, 
$\POPI_n=\POP_n\cap\I_n$ and $\PORI_n=\POR_n\cap\I_n$. 

\smallskip 

Concerning the monoids $\OP_n$ and $\POP_n$, Corollary 2.2,  Corollary 2.3, Theorem 2.4 and Theorem 2.6 of \cite{Fernandes:2023} depend on 
Proposition 2.1(1) of \cite{Fernandes:2023}, which is, although in other words, a transcription of Theorem 3 of \cite{Higgins&Vernitski:2022}. 
This last result, in turn, rewrites Proposition 1.1 of \cite{Levi&Mitchell:2006}, which was presented without proof. 

Although the statement of Theorem 3 of \cite{Higgins&Vernitski:2022} does not mention any restriction on the rank of the transformation 
$\alpha\in\T_n$ taken, in the proof the authors considered that $\alpha$ has rank greater than two. In fact, there are transformations of rank two which do not preserve orientation and, simultaneously, which satisfy the condition (for $\OP_n$) of Theorem 3 of \cite{Higgins&Vernitski:2022}. For example, $\alpha=\left(\begin{smallmatrix}1&2&3&4\\1&2&1&2\end{smallmatrix}\right)$, with $n=4$. On the other hand, clearly all transformations of rank one preserve orientation, so Theorem 3 of \cite{Higgins&Vernitski:2022} is only valid for transformations of rank other than two. We should note, however, that the statement of Proposition 1.1 of \cite{Levi&Mitchell:2006} does not lack the same restriction. 

\smallskip 

As a consequence of the foregoing, the statements of the mentioned results of \cite{Fernandes:2023} must comply with the same restriction as Theorem 3 of \cite{Higgins&Vernitski:2022}, i.e. they should be stated as follows. 

\begin{proposition2.1} {\em(1)}  
Let $\alpha\in\T_n$ be such that $|\im(\alpha)|\neq2$. 
Then, $\alpha\in\OP_n$ if and only if, for every triple $(a_1,a_2,a_3)$ of elements of $\Omega_n$, $(a_1,a_2,a_3)$ and 
$(a_1\alpha,a_2\alpha,a_3\alpha)$ are both cyclic or both anti-cyclic. 
\end{proposition2.1} 

\begin{corollary2.2}
Let $\alpha\in\T_n$ be such that $|\im(\alpha)|\neq2$. Then,  
$\alpha\in\OP_n$ if and only if, for every cyclic triple $(a_1,a_2,a_3)$ of elements of $\Omega_n$, $(a_1\alpha,a_2\alpha,a_3\alpha)$ is also cyclic.  
\end{corollary2.2} 

\begin{corollary2.3}
Let $\alpha\in\T_n$ be such that $|\im(\alpha)|\neq2$. Then,  
 $\alpha\in\OP_n$ if and only if, for every non-decreasing triple $(a_1,a_2,a_3)$ of elements of $\Omega_n$, 
the triple $(a_1\alpha,a_2\alpha,a_3\alpha)$ is cyclic. 
\end{corollary2.3}

Recall that the \textit{width} of a partial transformation is the number of elements in its domain. 

\begin{theorem2.4}
Let $\alpha\in\T_n$ be such that $|\im(\alpha)|\neq2$. Then,
 $\alpha\in\OP_n$  if and only if every restriction of $\alpha$ 
of width three belongs to $\POP_n$.  
\end{theorem2.4} 

\begin{theorem2.6}
Let $\alpha\in\PT_n$ be such that $|\im(\alpha)|\neq2$. Then, 
$\alpha\in\POP_n$ if and only if every restriction of $\alpha$ 
of width three belongs to $\POP_n$.  
\end{theorem2.6} 

\smallskip 

Now, let $\alpha\in\PT_n$ be such that $|\im(\alpha)|=2$. Suppose that $\dom(\alpha)=\{i_1<i_2<\cdots<i_k\}$ for some $2\leqslant k\leqslant n$. 
Then, it is clear that $\alpha\in\POP_n$ if and only if $\alpha$ admits as kernel classes $\{i_r,i_{r+1},\ldots,i_s\}$ and $\{i_1,\ldots,i_{r-1},i_{s+1}\ldots,i_k\}$ for some $1\leqslant r\leqslant s\leqslant k$ (and $s-r+1<k$). 

\smallskip 

We finish this note by observing that Theorem 2.7 of \cite{Fernandes:2023} does not require the same restriction (even for $\POPI_n$), since all partial permuations of rank two preserve orientation.

\bigskip 

\lastpage